\documentclass[10pt]{article}
\usepackage{amssymb,amsmath,color}
\usepackage{amsthm}
\usepackage{fullpage}
\usepackage{mathrsfs}
\usepackage{units}
\allowdisplaybreaks
\usepackage{color}

\theoremstyle{definition}

\let\oldsection\section
\renewcommand\section{\setcounter{equation}{0}\oldsection}
\renewcommand\thesection{\arabic{section}}

\newcommand{\be}{\begin{equation}\label}
\newcommand{\ee}{\end{equation}}
\newcommand{\beaa}{\begin{eqnarray}}
\newcommand{\bea}{\beaa\label}
\newcommand{\eea}{\end{eqnarray}}
\newcommand{\nn}{\nonumber}
\newcommand{\Om}{\Omega}

\newcommand{\pa}{\partial}
\newcommand{\ep}{\varepsilon}
\newcommand{\varphit}{\widetilde{\varphi}}

\newcommand{\Lp}{L^p(\Om)}
\newcommand{\Lq}{L^q(\Om)}
\newcommand{\Lin}{L^\infty(\Om)}

\newcommand{\norm}[2][ ]{\|#2\|_{#1}}
\DeclareMathOperator{\supp}{supp}

\newcommand{\R}{\mathbb{R}}

\newcommand{\io}{\int_\Om}

\newcommand{\fat}{\qquad \mbox{for all }t\in(0,T_{\max})}

\theoremstyle{plain}
\newtheorem{thm}{Theorem}[section]
\newtheorem{lem}[thm]{Lemma}
\newtheorem{cor}[thm]{Corollary}
\newtheorem{proposition}[thm]{Proposition}
\theoremstyle{definition}
\theoremstyle{remark}
\newtheorem{remark}[thm]{Remark}
\allowdisplaybreaks

\title{An interpolation inequality and its application in the Keller-Segel model}
\author{Xinru Cao \thanks{Institut f\"ur Mathematik, Universit\"at Paderborn, 
Warburger Str.100, 33098 Paderborn, Germany; \mbox{email: caoxinru@gmail.com}}
}

\begin{document}

\maketitle

\begin{abstract}
\noindent 
In this paper, we first prove an interpolation inequality of Ehrling-type, 
which is an improvement of a special case to the well known Gagliardo-Nirenberg inequality. Then we apply it to study the classical Keller-Segel system
\[
\left\{
\begin{array}{llc}
u_t=\Delta u-\nabla\cdot(u \nabla v), \\[6pt]
\displaystyle
v_t=\Delta v-v+u,
\end{array}
\right.
\]
in a bounded domain $\Omega\subset\R^N$ ($N\ge 2$) with smooth boundary. 
It is known that for any $\delta>0$, if 
$\io u^{\frac N2+\delta}(\cdot,t)$ is bounded, then the solution is 
global and bounded. 
Here we show that the same conclusion holds for a weaker assumption: the equi-integrability of 
$\{\io u^\frac N2(\cdot,t)|~t\in(0,T_{\max})\}$ can prevent blow up.
{\bf Math Subject Classification (2010):} 35K55, 35B40 35Q35, 92C17, 35B35
\end{abstract}

\section{Introduction}\label{introduction}
The following system called Keller-Segel model is proposed in 
\cite{keller1971model}
to model chemotatic migration 
\begin{equation}\label{eq:system}
\left\{
\begin{array}{rll}
&u_t=\Delta u-\nabla\cdot(u \nabla v), &(x,t)\in \Omega\times (0,T),\\[6pt]
\displaystyle
&v_t=\Delta v-v+u, &(x,t)\in\Omega\times (0,T),\\[6pt]
&\pa_\nu u=\pa_\nu v=0, &(x,t)\in \pa\Om\times(0,T),\\[6pt]
&u(x,0)=u_0(x),\quad  v(x,0)=v_0(x),
& x\in\Omega.
 \end{array}
\right.
\end{equation}
Here $\Om\subset \R^N$ ($N\ge 2$) is a bounded smooth domain, $T\in(0,\infty]$, and $\nu$ denotes the 
outer normal vector on $\pa\Om$.
Let $(u_0,v_0)$ be a nonnegative function pair,
$u$ and $v$ denote 
the density of cells and chemical 
concentration, respectively.
The system (\ref{eq:system}) describes an interesting interaction between the cells and
the chemical signal. 
This chemical substance is released by the cells themselves, and on the other hand,
it also attracts cells; meaning that the movement of 
cells is oriented to the higher density of chemical signal. 
The latter mechanism is known 
as chemotaxis, which is represented by 
the cross-diffusion
term $-\nabla\cdot(u\nabla v)$ in the first equation. 
This biological model plays an important role in numerous biological processes such as 
wound healing, cancer invasion. It also draws interests
from many mathematicians, for surveys in this area we refer to 
\cite{bellomo2015toward,horstmann20031970,hillen2009user} 
and the references therein.

A striking feature of this model is the occurrence of a blow up 
phenomenon caused by the
aggregation of cells, related research 
can be found in
\cite{herrero1997blow,horstmann2001blow,
nagai2000chemotactic,nagai2001blowup,
winkler2013finite,mizoguchi2014}.
The spatial dimension seems crucial in the mathematical analysis of detecting 
blow up.
In the one dimensional setting, blow up never happens. However, considering 
the two-dimensional case, 
one can prove the existence of radial blow up solutions 
if the initial data $(u_0,v_0)$ 
exceed the critical mass: $\io u_0>8\pi$ \cite{mizoguchi2014}; otherwise, 
the solution always remains bounded \cite{nagai1997application}.
In higher dimensions, whether a solution blows up does not depend on 
the total mass any more;  blow up solutions are constructed with 
any small mass \cite{winkler2013finite}.
On the other hand, looking for a sufficient condition which 
can prevent blow up may be 
of some {interest}, especially in two or higher dimensions. 

Throughout the paper, we consider the classical solution $(u,v)$ of \eqref{eq:system} 
on $\Om\times[0,T_{\max})$ emanating from the nonnegative initial pair 
$(u_0,v_0)\in C^0(\overline\Om)\times W^{1,\infty}(\Om)$, where $T_{\max}\in(0,\infty]$ 
denotes the maximal existence time of the solution. The local existence theory concerning this issue is presented in Lemma \ref{lem:locexist}.
Beyond this, a well known sufficient condition for global solutions is the following 
\cite[Lemma 3.2]{bellomo2015toward}:
\begin{proposition}
\label{lem:lp}
Let $N\ge 1$ and $p>\frac N2$. Assume that $\Om\subset \R^N$ is a bounded domain with 
smooth boundary and $(u,v)$ is a nonnegative classical solution of \eqref{eq:system} 
in $\Om\times(0,T_{\max})$ with maximal existence time $T_{\max}\in(0,\infty]$. If
\begin{align}\label{eq:lem1:lp}
\mathop{\sup}\limits_{t\in(0,T_{\max})} \norm[L^p(\Om)]{u(\cdot,t)}<\infty,
\end{align}
then 
\[
\mathop{\sup}\limits_{t\in(0,T_{\max})}\big(\norm[\Lin]{u(\cdot,t)}+
\norm[W^{1,\infty}(\Om)]{v(\cdot,t)}\big)<\infty.\]
\end{proposition}
The proof is carried out either by using Neumann heat
semigroup estimates or by studying a 
coupled energy evolution of $\io u^p$ and $\io |\nabla v|^{2q}$ with $p,q$ sufficiently 
large \cite{tao2012boundedness,freitag2016boundedness}.
Generally, the condition in the above proposition can not reach the 
borderline {value}
$p=\frac{N}{2}$. 
{In the special case when $N=2$ and thus $\frac N2=1$}, 
we already mentioned that blow up can
happen even though
 $\io u(\cdot,t)=\io u_0$ is bounded \cite{mizoguchi2014}. Therefore, we cannot expect {that} boundedness
 of $\norm[L^\frac{N}{2}(\Om)]{u(\cdot,t)}$ can prevent blow up.
 However, if we require a little more, {namely that
 $\{u^\frac{N}{2}(\cdot,t)\}_{t\in(0,T_{\max})}$ is not only bounded with respect to the spatial $L^1$-norm,
 but also enjoys an additional equi-integrability property}, 
 we will be able to show global existence and boundedness for the system. 
{Accordingly, the main result in the paper reads as follows:}

\medskip


\begin{thm}\label{thm:bdd}
 Assume that $\Om\subset\R^N$ ($N\ge 2$) is a bounded domain 
 with smooth boundary, {and that} the nonnegative initial data $(u_0,v_0)$
satisfy $u_0\in C^0(\overline{\Om})$ and $v_0\in W^{1,\infty}(\Om)$. 
Let $(u,v)$ be a nonnegative classical solution of \eqref{eq:system}
on $\Om\times(0,T_{\max})$ with maximal existence time $T_{\max}\in(0,\infty]$.
If
 \begin{align}
 \label{thm:con1}
 & \mathop{\sup}\limits_{t\in (0,T_{\max})} \norm[L^\frac{N}{2}(\Om)] {u(\cdot,t)}<\infty,\\
 \label{thm:con2}
 & \mbox{ and }\{u(\cdot,t)^\frac{N}{2}\}_{t\in(0,T_{\max})} \mbox{ is equi-integrable,}
 \end{align}
then  $(u,v)$ is global and bounded.
\end{thm}

Recalling De la Vall\'ee-Poussin Theorem, we 
obtain the following equivalent extension criterion:
\begin{cor}\label{chap1:corollary}
Assume that $(u,v)$ be a nonnegative classical solution of
\eqref{eq:system} on $\Om\times(0,T_{\max})$ with $T_{\max}\in(0,\infty]$.
Let  $f: [0,\infty) \to [0,\infty)$ be continuous and such that
\[
	\mathop{\lim}\limits_{s\to\infty} \frac{f(s)}{s^\frac N2}=\infty. 
\]
If we have
\begin{align}
\label{thm:con3}
\mathop{\sup}\limits_{t\in(0,T_{\max})}\io f(u(\cdot,t))<\infty,
\end{align}
then $(u,v)$  is global and bounded.
\end{cor}

The above corollary {inter alia} shows that
the boundedness of $\io u^\frac N2 \log u$ is sufficient 
for our conclusion, 
which is obviously not covered by Proposition \ref{lem:lp}.

On the other hand, Corollary \ref{chap1:corollary} also
improves the previous knowledge in the two-dimensional Keller-Segel model;
it is {known} that the
boundedness of $\io u\log u$ and $\io |\nabla v|^2$ can exclude blow up 
\cite[Lemma 3.3]{bellomo2015toward}. Now we can immediately remove the requirement on
$\io |\nabla v|^2$. Actually, in the simplified parabolic-elliptic 
system where 
the second equation in \eqref{eq:system} is replaced by $\Delta v-v+u=0$, 
a crucial elliptic estimate shows that the boundedness of $\io |\nabla v|^2$
already results from the boundedness of $\io u\ln u$ \cite[Lemma A.4]{tao2014energy}. Thus 
we know the solution is bounded only if
 $\io u\ln u$ is bounded without applying the current result. 
However, {since a} corresponding estimate for $\io |\nabla v|^2$ in a parabolic equation {appears to be}
lacking, the outcome of {the above} corollary seems not trivial in the fully parabolic model. 
Moreover, the condition can be weakened to the boundedness of the
$L^1$-norm of {essentially any superlinear functional}
of $u$, e.g. $\io u\log\log {(u+e)}$.

Additionally, by virtue of an equivalent definition of equi-integrability,
Theorem \ref{thm:bdd} can be  {rephrased} in the following way:
\begin{cor}
Let $(u,v)$ be a classical solution of \eqref{eq:system} on $\Om\times(0,T_{\max})$  with $T_{\max}\in(0,\infty]$.
For all $\ep>0$ there is $\delta>0$ such that for any measurable set $E\subset \Om$ {with}
$|E|<\delta$, 
if we have
\[\mathop{\sup}\limits_{t\in(0,T_{\max})}\int_E u^\frac{N}{2}(\cdot,t)<\ep,\] 
then 
\[\mathop{\sup}\limits_{t\in(0,T_{\max})} \norm[\Lin]{u(\cdot,t)}<\infty.\]
\end{cor}

We {note that this property resembles the feature of } $\ep$-regularity derived in 
\cite{sugiyama2010varepsilon}
for a porous medium type parabolic-elliptic Keller-Segel model in the 
whole space or for a corresponding degenerate fully parabolic system in a bounded domain \cite{ishige2017}. 
Since our result in the above corollary is independent of time, {this analogy is further underlined in the following consequence describing
the behavior of unbounded solutions, which also applys infinite time blow-up.}
\begin{thm}\label{cor:blowup}
 Assume that $\Om\subset\R^N$ ($N\ge 2$) is a bounded domain 
 with smooth boundary.
 Let $(u,v)$ be a classical solution of \eqref{eq:system} on $\Om\times(0,T_{\max})$ 
 with $T_{\max}\in(0,\infty]$.
Suppose that 
\[
\mathop{\sup}\limits_{t\in(0,T_{\max})}\norm[\Lin]{u(\cdot,t)}=\infty.
\] 
Then $\{u^\frac N2(\cdot,t)\}_{t\in(0,T_{\max})}$ is 
not equi-integrable. In other words, 
there are $\ep_0>0$, and $x_0\in\Om$ such that for all $\rho>0$,
\[
\mathop{\sup}\limits_{t\in(0,T_{\max})} 
\int_{B_\rho(x_0) \cap \Om} u^{\frac N2}(\cdot,t)>\ep_0.
\]
\end{thm}
\section{An interpolation inequality}
In the analysis of chemotaxis models, the Gagliardo-Nirenberg inequality is frequently used, especially in
the style of the following form
\begin{align}\label{chap1:gn}
 \norm[L^q(\Om)]{\varphi} \le 
  C_1\norm[L^r(\Om)]{\nabla \varphi}^{a} \norm[\Lp]{\varphi}^{1-a}
   +C_2 \norm[L^{p}(\Om)]{\varphi} \mbox{ for all } \varphi\in W^{1,r}(\Om),
 \end{align}
 where $a={\frac{\frac{N}{p}-\frac{N}{q}}{1-\frac Nr+\frac{N}{p}}} \in (0,1)$ \cite[Theorem 10.1]{friedman_book}. 
 Here the constant $C_1>0$ depends on $p,q,r$ and $\Om$. 
 When applying the Gargliardo-Nirenberg inequality, we usually require the 
 exponent $a$ {to be} strictly less
than a given power in order to control a target term. 
One can imagine that if $C_1>0$ could be chosen arbitrarily small, 
we would be able to deal with more subtle critical cases 
\cite{biler1994debye}.

The purpose of this section is to investigate a kind of interpolation inequality 
with the aforementioned ambition
that the constant $C_1$ can be arbitrarily small. However, this is not generally true. 
Following the idea from \cite[Lemma 5.1]{lou2015global},
we actually show that such {an}
interpolation inequality holds
for the class of 
equi-integrable functions. 
This
is similar to that of  \cite[Theorem 3]{biler1994debye} 
and \cite[Lemma 5.1]{lou2015global}. 
\begin{lem}\label{lem:interpolation1} 
 Let $\Om\subset \R^N$ be bounded with smooth boundary. Let
 $r\ge1$, $0<q< \frac{Nr}{(N-r)_+}$.  For any $0<\theta<q$, we define
 \begin{equation}\label{chap1:pq}
 \begin{array}{llc}
 &
 p:=
 \left\{
 \begin{array}{llc}
 &{N(\frac qr-1)}, & \mbox{ if } q>r,\\
 &\theta ,& \mbox{ if } q\le r,
 \end{array}
 \right.
 &
q_0:=
 \left\{
 \begin{array}{llc}
 & q, &  \mbox{ if } q>r,\\
 & r(1+\frac{p}{N}),& \mbox{ if } q\le r.
 \end{array}
 \right.
 \end{array}
 \end{equation}
 \[a:= {\frac{\frac{N}{p}-\frac{N}{q}}{1-\frac Nr+\frac{N}{p}}}\in (0,1), 
 \quad b:={\frac{\frac{1}{p}-\frac{1}{q}}
{\frac{1}{p}-\frac{1}{q_0}}}\in(0,1].
 \]
Let $\delta:(0,1)\to (0,\infty)$ be nondecreasing.
Then for each $\ep>0$, we can find $C_\ep>0$ such that
 \begin{align}\label{eq:interpolation}
   \norm[L^q(\Om)]{\varphi} \le 
   \ep\norm[L^r(\Om)]{\nabla \varphi}^{a} \norm[\Lp]{\varphi}^{1-b}
   +C_\ep \norm[L^{p}(\Om)]{\varphi}^{(1-\frac Nr+\frac{N+r}{q_0})b+(1-b)}
   +C_\ep \norm[L^{p}(\Om)]{\varphi}
   +C_\ep  \norm[L^{p}(\Om)]{\varphi}^{1-b}.
 \end{align}
is valid for any
\begin{align}\nn
\varphi\in \mathcal F_\delta:=\bigg\{\psi\in W^{1,r}(\Om) \ \bigg| \ 
\mbox{ For all } \ep'\in(0,1), &\mbox{ we have } \int_{E} \psi^p<\ep'  \mbox{ for all measurable {sets} }\\
\label{chap1:f_delta}&
E \subset \Om \mbox{ with } |E|<\delta(\ep')
 \bigg\}.
\end{align}
\end{lem}

\begin{proof}
We first consider the case $q>r$, hence $\frac qr-1>0$. 
We abbreviate $s:=\frac{Nr}{N+r}<\min\{N,r\}$. {Then} according to the {Sobolev} 
embedding:
$W_0^{1,s}(\R^N)\hookrightarrow L^r(\R^N)$, there is a constant $c_1>0$ such that
\begin{align}\label{eq:sobolev}
\norm[L^r(\R^N)]{\psi}\le c_1\norm[L^s(\R^N)]{\nabla\psi}
\end{align}
for all $\psi\in W_0^{1,s}(\R^N)$. Let $\Om'$ be a bounded open set such that
$\Om \subseteq \Om'$.
In light of Theorem \ref{thm:extension}, we can find $c_2>0$ and 
extend $\varphi\in W^{1,r}(\Om)$ to  $\widetilde{\varphi}\in W^{1,r}_0(\R^N)$
in such a way that 
 \begin{align}\nn
 &{\varphit}=\varphi \mbox{ a.e.  in }\Om,
 \quad \supp {\varphit} \subset \Om',\\\label{chap1:tildephi1}
 &\norm[L^q(\Om')]{{\varphit}}\le c_2\norm [L^q(\Om)]{{\varphi}},
 \quad  \norm[L^r(\Om')]{\nabla{\varphit}}^r
 \le c_2\norm [L^r(\Om)]{{\nabla\varphi}}^r,
 \end{align}
and that
there is a nondecreasing function $\widetilde\delta:(0,1)\to(0,\infty)$ such that
\begin{align}\nn
\varphit\in\mathcal F_{\widetilde\delta}:=\bigg\{\psi\in W^{1,r}(\Om) \ \bigg| \ 
\mbox{ For all } \ep'\in(0,1), &\mbox{ we have } \int_{E} \psi^p<\ep'  \mbox{ for all measurable sets }\\
&
E \subset \Om' \mbox{ with } |E|<\widetilde\delta(\ep')
 \bigg\}.
\end{align}
Given $\ep>0$, 
let $\ep':=\left(\frac{\ep^q }{2^r (\frac qr)^rc_1 c_2}\right)^{\frac Nr}$
and let $\delta:=\widetilde\delta(\ep')>0$. We have 
 \bea{tildephi_equi}
 \int_{B} |\varphit|^{p}<\ep'
 \eea
 for any ball $B\subset\Om'$ and with radius 
 no bigger
 than 
 $\eta:=\left(\frac{\delta}{w_n}\right)^\frac1N$, 
 where $w_n$ denotes the volume of the unit ball in $\R^N$.

Since $\Om$ is bounded, we can find a family of finite balls $\{B_j\}_{1\le j\le M}$ 
 with radius larger than $\eta$ to cover $\overline{\Om}$ with
 $\overline{\Om}\subset \mathop{\cup}\limits_{1\le j\le M} B_j \subseteq \Om'$. 
Moreover, there exists $c_3>0$ and a smooth partition of unity 
for $\mathop{\cup}\limits_{1\le j\le M} B_j$ is given by 
 a family of nonnegative functions $\{\zeta_j\}_{1\le j\le M}$ satisfying
 \begin{align}\label{eq:zeta}
 \supp \zeta_j\subset B_j,\quad |\nabla {\zeta_j}^\frac{1}{r}|< \frac{c_3}{\eta} \mbox{ for all } 1\le j\le M, \mbox{ and}
 \quad \mathop{\sum}\limits_{j=1}^{j=M} \zeta_j=1.
 \end{align}
 
It can be easily checked that  ${\varphit}^{\frac{q}{r}} \zeta_j^\frac{1}{r} 
\in W_0^{1,s} (B_j)$ since $\varphit\in W_0^{1,r}(\R^N)$ and $q<\frac{Nr}{(N-r)_+}$, therefore
we can invoke
(\ref{eq:sobolev}) and the elementary inequality
\[ 
(a+b)^s\le 2^{s-1} a^s+2^{s-1} b^s   \mbox{ for all $s>1$ and $a,b>0$},
\]
to obtain that
 \begin{align}\nn
  \int_{\Om'} \widetilde{\varphi}^q\zeta_j&=\norm[L^r(B_j)]{\varphit^{\frac{q}{r}} 
  \zeta_j^\frac{1}{r}}^r\\\nn
  &\le c_1\norm[L^s(B_j)]{\nabla(\varphit^{\frac{q}{r}} \zeta_j^\frac{1}{r})}^r\\\nn
  &\le c_1\norm[L^s(B_j)]{\frac{q}{r} 
  \varphit^{\frac{q}{r}-1}\zeta_j^\frac{1}{r}\nabla\varphit
  +\varphit^\frac{q}{r}\nabla\zeta_j^\frac{1}{r}}^r\\\label{lem:1.1}
  &\le c_12^{r-1}(\frac qr)^r\left(\int_{B_j} 
  |\varphit^{\frac{q}{r}-1}\zeta_j^\frac{1}{r}\nabla\varphit|^s\right)^\frac{r}{s}
  +c_12^{r-1}(\frac{c_3}{\eta})^r
  \left(\int_{B_j}\varphit^{s\frac qr}\right)^\frac{r}{s}.
  \end{align}
  On applying H\"older's inequality and \eqref{tildephi_equi}, the first term on the right-hand side of \eqref{lem:1.1} can be estimated as
  \begin{align}
  \nn
    c_12^{r-1}(\frac qr)^r\left(\int_{B_j} 
  |\varphit^{\frac{q}{r}-1}\zeta_j^\frac{1}{r}\nabla\varphit|^s\right)^\frac{r}{s}
  &\le  c_12^{r-1}(\frac qr)^r\left(\int_{B_j}
  |\varphit|^{s(\frac{q}{r}-1)\frac{r}{r-s}}\right)^{\frac{r}{s}-1}
  \left(\int_{\Om'}\zeta_j|\nabla\varphit|^r\right)
  \\\nn
  &=  c_12^{r-1}(\frac qr)^r\left(\int_{B_j} 
  |\varphit|^{N(\frac qr-1)}\right)^{\frac rs-1} \int_{\Om'}\zeta_j|\nabla\varphit|^r\\\nn
  &=  c_12^{r-1}(\frac qr)^r\left(\int_{B_j} |\varphit|^{p}\right)^{\frac rN} \int_{\Om'}\zeta_j|\nabla\varphit|^r
  \\\label{lem:1.2}
  &\le  c_12^{r-1}(\frac qr)^r(\ep')^{\frac rN}  \int_{\Om'}\zeta_j|\nabla\varphit|^r\le \frac{\ep^q}{2c_2} \int_{\Om'}\zeta_j|\nabla\varphit|^r.
 \end{align}
Now we claim that for all $r<q<\frac{Nr}{(N-r)_+}$, there are constants $c_\ep,c_4,c_5>0$ such that
\begin{align}\label{lem:1.3'}
c_12^{r-1}(\frac{c_3}{\eta})^r
  \left(\int_{B_j}|\varphit|^{s\frac qr}\right)^\frac{r}{s}
  \le \frac{\ep^q}{2c_2M}\int_{\Om'}|\nabla\varphit|^r
 +c_\ep \norm[L^{N(\frac qr-1)}(\Om')]{\varphit}^{q-\frac{Nq}{r}+N+r}
 +c_4\norm[L^{N(\frac qr-1)}(\Om')]{\varphit}^\frac{N+r}{N}+c_5\ep^{\frac {q^2}{q-r}}.
\end{align}

If $r<q<\frac{r(N+r)}{N}$, 
let 
$d=\frac{\frac{1}{\frac qr-1}-
\frac{Nr}{sq}}{1-\frac{N}{r}+\frac{1}{\frac qr-1}}$, hence $d\in(0,1)$.  Moreover, since $s<r$, we know that 
 $aq<r$. The Gagliardo-Nirenberg inequality thus implies the existence of $c_4>0$ and $c_\ep>0$ such that
 \begin{align}\nn
 2^{r-1}c_1(\frac{c_3}{\eta})^r\left(\int_{B_j} 
 |\varphit|^\frac{sq}{r}\right)^\frac{r}{s}
  &\le 2^{r-1}c_1(\frac{c_3}{\eta})^r\norm[L^\frac{sq}{r}(\Om')]{\varphit}^q\\\nn
  &\le c_4\norm[L^r(\Om')]{\nabla\varphit}^{dq}
  \norm[L^{N(\frac qr-1)}(\Om')]{\varphit}^{(1-d)q}
  +c_4\norm[L^{N(\frac qr-1)}(\Om')]{\varphit}^q\\\nn
  &= c_4\left(\int_{\Om'}|\nabla \varphit|^r\right)^\frac{dq}{r}
  \left( \int_{\Om'} \varphit^{N(\frac qr-1)}
  \right)^{\frac{(1-d)q}{N(\frac qr-1)}}+c_4\norm[L^{N(\frac qr-1)}(\Om')]{\varphit}^q\\\label{lem:1.3}
  &\le \frac{\ep^q}{2c_2M}\int_{\Om'}|\nabla\varphit|^r
  +c_\ep \norm[L^{N(\frac qr-1)}(\Om')]{\varphit}^{q-\frac{Nq}{r}+N+r}
  +c_4\norm[L^{N(\frac qr-1)}(\Om')]{\varphit}^q.
 \end{align}
 
If $ \frac{r(N+r)}{N}\le q <\frac{Nr}{(N-r)_+}$, hence $\frac{sq}{r}\le N(\frac qr-1)$. We can simply use H\"older's inequality to obtain a constant $c_5>0$ fulfilling
\begin{align}\nn
2^{r-1}c_1(\frac{c_3}{\eta})^r
  \left(\int_{B_j}|\varphit|^{s\frac qr}\right)^\frac{r}{s}
  &\le 
  2^{r-1}c_1(\frac{c_3}{\eta})^r |\Om'|^{1-\frac{sq}{pr}}
   \left(\int_{B_j}|\varphit|^{p}\right)^\frac{q}{p}\le c_5\ep^{\frac {q^2}{q-r}}.
\end{align}
Hence \eqref{lem:1.3'} holds for all $r<q<\frac{Nr}{(N-r)_+}$.
Combining (\ref{lem:1.1}-\ref{lem:1.3'}), we see that for each $1\le j\le M$,
\begin{align}\nn
 \int_{\Om'}|\varphit|^q\zeta_j
& \le \frac{\ep^q}{2c_2} \int_{\Om'}\zeta_j|\nabla\varphit|^r
 +\frac{\ep^q}{2c_2M} \int_{\Om'}|\nabla\varphit|^r\\
 \label{eq:lem:1.4}
 &~~~~~~~~~~
 +c_\ep \norm[L^{N(\frac qr-1)}(\Om')]{\varphit}^{q-\frac{Nq}{r}+N+r} 
 +c_4\norm[L^{N(\frac qr-1)}(\Om')]{\varphit}^q
 +c_5\ep^{\frac {q^2}{q-r}}.
\end{align}
Finally, we obtain from \eqref{eq:lem:1.4} and \eqref{eq:zeta} that 
\begin{align}\nn
 &~~~~ \norm[L^q(\Om)]{\varphi}^q \\\nn
  &\le \norm[L^q(\Om')]{\varphit}^q
  =\int_{\Om'} |\varphit|^q\left(\mathop{\sum}\limits_{j=1}^{j=M} \zeta_j\right)
 =\mathop{\sum}\limits_{j=1}^{j=M}  \int_{\Om'} 
 |\varphit|^q\zeta_j\\\nn
 &\le \mathop{\sum}\limits_{j=1}^{j=M}  
 \left(\frac{\ep^q}{2c_2} \int_{\Om'}\zeta_j|\nabla\varphit|^r
 +\frac{\ep^q}{2c_2M}\int_{\Om'}|\nabla\varphit|^r
 +c_\ep \norm[L^{N(\frac qr-1)}(\Om')]{\varphit}^{q-\frac{Nq}{r}+N+r}
 +c_4\norm[L^{N(\frac qr-1)}(\Om')]{\varphit}^q
 +c_5\ep^{\frac {q^2 }{q-r}}\right)\\\nn
 &\le
 \frac{\ep^q}{2c_2} \int_{\Om'}|\nabla\varphit|^r
 \left( \mathop{\sum}\limits_{j=1}^{j=M} \zeta_j\right)
 +M\left(\frac{\ep^q}{2c_2M}\int_{\Om'}|\nabla\varphit|^r
 +c_\ep \norm[L^{N(\frac qr-1)}(\Om')]{\varphit}^{q-\frac{Nq}{r}+N+r}
 +c_4\norm[L^{N(\frac qr-1)}(\Om')]{\varphit}^q
 +c_5\ep^{\frac {q^2}{q-r}}\right)\\\nn
  &\le \frac{\ep^q}{2c_2}\int_{\Om'}|\nabla\varphit|^r
  +\frac{\ep^q}{2c_2}\int_{\Om'}|\nabla\varphit|^r
  + c_\ep M\norm[L^{N(\frac qr-1)}(\Om')]{\varphit}^{q-\frac{Nq}{r}+N+r} 
  +c_4M\norm[L^{N(\frac qr-1)}(\Om')]{\varphit}^q+c_5M\ep^{\frac {q^2}{q-r}}\\\nn
 & \le \ep^q\int_{\Om}|\nabla\varphi|^r
 + C_\ep \norm[L^{N(\frac qr-1)}(\Om)]{\varphi}^{q-\frac{Nq}{r}+N+r} 
 +C_\ep\norm[L^{N(\frac qr-1)}(\Om)]{\varphi}^q+C_\ep
\end{align}
with some constant $C_\ep >0$.
Note that $b=1$ if $q>r$, taking the $q$-th root on both sides leads to \eqref{eq:interpolation} for the case $q>r$.

If $q\le r$, 
we see that $q_0>r\ge q>\theta$.
The H\"older inequality with 
$b=\frac{\frac{1}{\theta}-\frac{1}{q}}{\frac{1}{\theta}-\frac{1}{q_0}}$ 
shows that
\begin{align}\label{lem:1.4}
\norm[\Lq]{\varphi}
\le
\norm[L^{q_0}(\Om)]{\varphi}^b\norm[L^{\theta}(\Om)]{\varphi}^{1-b}.
\end{align}
Since $q_0>r$, we have already proven that for all $\ep>0$, there is $C_\ep>0$ so that
\begin{align}\nn
\norm[L^{q_0}(\Om)]{\varphi} &\le \left(\ep^\frac{q_0}{b} \norm[L^r(\Om)]{\nabla\varphi}^r
+C_{\ep} \norm[L^\theta(\Om)]{\varphi}^{q_0-\frac{Nq_0}{r}+N+r}
+C_{\ep} \norm[L^\theta(\Om)]{\varphi}^{q}+C_\ep
\right)^{\frac{1}{q_0}}\\\nn
&\le 
\ep^\frac{1}{b} \norm[L^r(\Om)]{\nabla\varphi}^{\frac{r}{q_0}}
+C_{\ep} \norm[L^\theta(\Om)]{\varphi}^{1-\frac{N}{r}+\frac{N+r}{q_0}}
+C_{\ep} \norm[L^\theta(\Om)]{\varphi}
+C_\ep,
\end{align}
which combined with the previous interpolation inequality \eqref{lem:1.4} yields that
\begin{align}\nn
 \norm[\Lq]{\varphi}
 &\le \left(\ep^\frac{1}{b} \norm[L^r(\Om)]{\nabla\varphi}^{\frac{r}{q_0}}
+C_\ep \norm[L^\theta(\Om)]{\varphi}^{1-\frac{N}{r}+\frac{N+r}{q_0}}
+C_\ep \norm[L^\theta(\Om)]{\varphi}
+C_\ep
\right)^b \norm[L^{\theta}(\Om)]{\varphi}^{1-b}.
\\\nn
&\le 
\ep \norm[L^r(\Om)]{\nabla\varphi}
^{b\cdot\frac{r}{q_0}} 
\norm[L^\theta(\Om)]{\varphi}^{1-b}
+C_\ep \norm[L^\theta(\Om)]{\varphi}^{(1-\frac Nr+\frac{N+r}{q_0})b+1-b}
+C_\ep \norm[L^\theta(\Om)]{\varphi}
+C_\ep \norm[L^\theta(\Om)]{\varphi}^{1-b}
\end{align}
We easily check that 
$b\cdot\frac{r}{q_0}=\frac{\frac{N}{\theta}-\frac{N}{q}}{1-\frac Nr+\frac{N}{\theta}}$,
thus \eqref{eq:interpolation} is valid for $q\le r$ as well.
\end{proof}

\begin{remark}
 The exponent $a$ in \eqref{eq:interpolation} is exactly the one from the Gagliardo-Nirenberg inequality
 \[
 \norm[L^q(\Om)]{\varphi} \le 
  C\norm[L^r(\Om)]{\nabla \varphi}^{a} \norm[\Lp]{\varphi}^{1-a}
   +C \norm[L^{p}(\Om)]{\varphi} \mbox{ for all } \varphi\in W^{1,r}(\Om).
 \]
 However $1-b \neq 1-a$. In fact, following the proof we can find $a+1-b<1$.
\end{remark}

\begin{remark}
 Given a family of functions $\{f_j\}_{j\in\mathbb{N}}$ 
 such that $\{f_j^p\}_{j\in\mathbb{N}}$ is equi-integrable, there exists 
 $\delta:(0,1)\to (0,\infty)$ nondecreasing such that
$f_j\in\mathcal{F}_\delta$, where $\mathcal{F}_\delta$ is defined in \eqref{chap1:f_delta}.
Therefore, we can apply Lemma \ref{lem:interpolation1} to a family of functions 
enjoying
equi-integrability.
\end{remark}

\section{Preliminaries for the Keller-Segel model}
\label{sec:pre}
In this section, some basic knowledge on the Keller-Segel system is prepared. We first introduce the well-established 
 local existence theory 
for \eqref{eq:system}. 
The proof can be found in many previous work, 
e.g. \cite[Lemma 3.1]{bellomo2015toward}.
\begin{lem}\label{lem:locexist}
Assume that $\Om\subset\R^N$ ($N\ge 2$) is a bounded domain 
with smooth boundary
that the initial data $(u_0,v_0)$ are nonnegative and
satisfy $u_0\in C^0(\overline{\Om})$ and $v_0\in W^{1,\infty}(\Om)$. 
There exists $T_{\max}\in(0,\infty]$ 
with the property such that the problem possesses a unique nonnegative classical 
solution $(u,v)$ 
satisfying 
\begin{align}\nn
&u\in C^0(\overline{\Om}\times[0,T_{\max}))\cap 
C^{2,1}(\overline{\Om}\times(0,T_{\max})),\\\nn
&v\in C^0(\overline{\Om}\times[0,T_{\max}))\cap 
C^{2,1}(\overline{\Om}\times(0,T_{\max}))\cap
L^\infty_{loc}([0,T_{\max});W^{1,\infty}(\Om)). \end{align}
Moreover, if $T_{\max}<\infty$, then 
\[
\norm[\Lin]{u(\cdot,t)}+\norm[W^{1,\infty}(\Om)]{v(\cdot,t)}\to 
\infty, 
\text { as }t\to T_{\max}.
\]
\end{lem}

The following properties can be easily checked. 
\begin{lem}\label{lem:l1}
We have
\begin{align}
&\io u(\cdot,t)=\io u_0, \\
&\io v(\cdot,t)\le \max \left\{\io v_0,\io u_0\right\}, \mbox{ for all }t\in(0,T_{\max}).
\end{align}
\end{lem}

In order to deal with a kind of spatial derivative estimate involving 
a time potential function, we introduce
the following version of maximal Sobolev regularity, which 
has been used in \cite[Lemma 2.5]{cao2016boundedness} and
\cite{yang2015boundedness}.
\begin{lem}\label{chap1:lem:maximal}
Let $r,q\in(1,\infty)$, and $T\in(0,\infty]$, $f\in L^r((0,T);L^q(\Om))$.
Let $v$ be the unique strong solution to the following evolution equation
\begin{eqnarray}\label{} 
\left\{
\begin{array}{llll}\displaystyle  
&v_t=\Delta v-v+f, &(x,t)\in\Om\times(0,T)\\[4pt]
\displaystyle
&\partial_\nu v=0, &(x,t)\in\pa\Om\times(0,T)\\[4pt]
\displaystyle 
&v(x,0)= v_0(x), & x\in \Om.\\[4pt]
\end{array}\right. \end{eqnarray}

There exists $C>0$, such that if $t_0\in[0,T)$, $v(\cdot,t_0)$ 
satisfies $v(\cdot,t_0)\in W^{2,r}(\Omega)$ with $\partial_\nu
v(\cdot,t_0)=0$, we have 
\begin{align}\label{maximal}
\int_{t_0}^T e^{\frac{rt}{2}}\norm[L^q(\Om)]{\Delta v(\cdot,t)}^r ds
\le C\int_{t_0}^T e^{\frac{rt}{2}}\norm[L^q(\Om)]{f(\cdot,t)}^r ds+
C e^{ \frac{rt_0}{2}}\norm[W^{2,q}(\Om)]{v(\cdot,t_0)}^r,
\end{align}
where $C$ depends on $q,r,\Om$.
\end{lem}

\begin{proof}
For given $t_0\in(0,T)$, we know that $\pa_\nu v(\cdot,t_0)=0$ on $\pa\Om$. 
Let $d:=\min\{\frac {T-t_0}{4},1\} $ and 
let $\chi\in C_0^\infty([0,\infty))$ be a cut-off function 
 satisfying 
 \begin{equation}
 \left\{
 \begin{array}{lll}
 & \chi(s)=1, & s=0,\\
 & \chi(s)\le 1, & 0 <s<d,\\
 & \chi(s)=0, & s\ge d.
 \end{array}
 \right.
 \end{equation}
Moreover,  $|\chi'(s)|\le \frac 2d$ for all $s\in[0,\infty)$. 
Let $w(x,s):=e^{\frac12 s}v(x, s+t_0)-\chi(s)v(x,t_0)$ 
for $(x,s)\in \Om\times[0,{T-t_0})$. We see that $w$ solves
the following equation
\begin{equation}\left\{
 \begin{array}{lll}\displaystyle
& w_s(x,s)=(\Delta-\frac 12) w(x,s)+e^{\frac12 s} f(x, s+t_0)+g(x,s), 
&(x,s)\in\Om\times(0,{T-t_0}),\\[4pt]
   \displaystyle
   &\nabla w\cdot\nu=0,&(x,s)\in\pa\Om\times[0,{T-t_0}),\\[4pt]
   \displaystyle
   & w(x,0)=0, &x\in\Om,
 \end{array}
\right.
\end{equation}
where $g(x,s):=\chi(s)\Delta v(x,t_0)-\chi'(s)v(x,t_0)
-\frac12 \chi(s)v(x,t_0)$ in $\Om\times[0,{T-t_0})$.

An application of the maximal Sobolev regularity result 
from  \cite{giga1991abstract}
implies the existence of $C_{q,r}>0$ such that
\begin{align*}
&\int_0^{T-t_0} \norm[\Lq]{\Delta w(\cdot,s)}^r ds \\
\le& C_{q,r}\int_0^{T-t_0} 
\norm[\Lq]{ e^{\frac12 s} f(x, s+t_0)}^r ds\\
&~~~~~~~~~~
+C_{q,r} \int_0^{T-t_0}
\norm[\Lq]{\chi(s)\Delta v(x,t_0)-\chi'(s)v(x,t_0)-\frac12 \chi(s)v(x,t_0)}^r ds\\
\le & 
C_{q,r}\int_0^{T-t_0} \norm[\Lq]{ e^{\frac12 s} f(x, s+t_0)}^r ds
+3^{r-1}C_{q,r} d (\frac 2d+\frac{3}{2}) \norm[W^{2,q}(\Om)]{v(\cdot,t_0)}^r
\\
\le& C_{q,r}\int_0^{T-t_0} \norm[\Lq]{ e^{\frac12 s} f(x, s+t_0)}^r ds
+4^{r}C_{q,r} \norm[W^{2,q}(\Om)]{v(\cdot,t_0)}^r.
\end{align*}

Since $e^{\frac12 s} \Delta v(x, s+t_0)=\Delta w(x,s)+\chi(s)\Delta v(x,t_0)$, we have
\begin{align*}
&\int_0^{T-t_0}e^{\frac{rs}{2}}\norm[\Lq]{\Delta v(\cdot, s+t_0)}^r ds\\
\le& 2^{r-1} \int_0^{T-t_0} \norm[\Lq]{\Delta w(\cdot,s)}^r ds+
2^{r-1} \int_0^{T-t_0} \norm[\Lq]{\chi(s)\Delta v(\cdot,t_0)}^r\\
\le&2^{r-1}C_{q,r}\int_0^{T-t_0} \norm[\Lq]{ e^{\frac12 s} f(x, s+t_0)}^r ds
+2^{r-1}(4^rC_{q,r}+1) \norm[W^{2,q}(\Om)]{v(\cdot,t_0)}^r.
\end{align*}
Upon changing variables, 
we obtain that
\begin{align}\nn
&\int_{t_0}^T e^{\frac{r}{2}(t-t_0)}\norm[\Lq]{\Delta v(\cdot,t)}^r dt \\\label{pre}
\le& 2^{r-1}C_{q,r} \int_{t_0}^T  e^{\frac{r}{2}(t-t_0)} \norm[\Lq]
{f(\cdot,t)}^r dt
+(8^rC_{q,r}+2^{r-1}) \norm[W^{2,q}(\Om)]{v(\cdot,t_0)}^r,
\end{align}
where 
\eqref{maximal} follows by multiplying \eqref{pre} by 
$ e^{\frac{r}{2}t_0}$ and choosing $C:=8^{r}C_{q,r}+2^{r-1}$.
\end{proof}
%
\section{Proof of Theorem \ref{thm:bdd}}\label{sec:proof}
Having in hand Proposition \ref{lem:lp}, 
we see that it is sufficient to show that
\eqref{eq:lem1:lp} holds for some $p>\frac N2$.
Before going into details, let us first prepare the following embedding lemma.
\begin{lem}\label{lem:embedding}
Let $\Om\subset \R^N$ be a bounded domain with smooth boundary, and 
let $\alpha\in(1,N)$.
For all $s\in(0,\infty]$, there is $C>0$ such that
\begin{align}\label{lem:emb0}
 \norm[L^\frac{N\alpha}{N-\alpha}(\Om)]{\nabla\varphi}
 \le C\norm[L^\alpha(\Om)]{\Delta \varphi}
 +C\norm[L^s(\Om)]{\varphi}, \mbox{ for all } \varphi \in W^{2,\alpha}(\Om) \mbox{ with }
 \pa_\nu \varphi=0 \mbox{ on } \pa\Om.
\end{align}
\end{lem}

\begin{proof}
 Using the fact that with some $c_1>0$, the estimates 
 $\norm[W^{2,\alpha}(\Om)]{\varphi}
 \le c_1(\norm[L^\alpha(\Om)]{\varphi}
 +\norm[L^\alpha(\Om)]{\Delta\varphi}$) holds for all $\varphi\in W^{2,\alpha}(\Om)$ 
 with $\pa_\nu \varphi|_{\pa\Om}=0$ \cite[Theorem 19.1]{friedman_book}, 
 we obtain
 a constant $c_2>0$ from the embedding  $W^{2,\alpha}(\Om) 
 \hookrightarrow 
 W^{1,\frac{N\alpha}{N-\alpha}}(\Om)$ that
 \begin{align}\label{lem:emb1}
 \norm[L^\frac{N\alpha}{N-\alpha}(\Om)]{\nabla\varphi}
 \le c_2(\norm[L^\alpha(\Om)]{\Delta\varphi}
 +\norm[L^\alpha(\Om)]{\varphi}).
 \end{align}

If $s<\alpha$, let 
$b=\frac{\frac Ns-\frac{N}{\alpha}}{2+\frac{N}{s}-\frac{N}{\alpha}}\in(0,1)$.
The Gargliardo-Nirenberg 
inequality together with Ponincar\'e inequality and Young's inequality implies
\begin{align}\nn
\norm[L^\alpha(\Om)]{\varphi}
&\le c_3\norm[L^\frac{N\alpha}{N-\alpha}(\Om)]{\nabla\varphi}^b
\norm[L^s(\Om)]{\varphi}^{1-b}
+c_3\norm[L^s(\Om)]{\varphi}\\\label{lem:emb3}
&\le \frac1{2c_2}  \norm[L^\frac{N\alpha}{N-\alpha}(\Om)]{\nabla\varphi}
+c_4 \norm[L^s(\Om)]{\varphi}
\end{align}
with some constant $c_3,c_4>0$ for all $\varphi\in W^{2,\alpha}(\Om)$ 
 with $\pa_\nu \varphi|_{\pa\Om}=0$.
If $s\ge \alpha$, we use H\"older's inequality
\begin{align}\label{lem:emb4}
\norm[L^\alpha(\Om)]{\varphi}\le |\Om|^{1-\frac{\alpha}{s}}\norm[L^s(\Om)]{\varphi}
\end{align}
instead of \eqref{lem:emb3}. Collecting (\ref{lem:emb1}-\ref{lem:emb4})
together yields \eqref{lem:emb0}.
\end{proof}

Now we are in a position to proceed the proof of our main ingredient. 

\begin{lem}\label{lem:n}
Assume that $\Om\subset\R^N$ ($N\ge 2$) is a bounded domain with smooth boundary,
Let $(u,v)$ be a classical solution of \eqref{eq:system} on
$\Om\times(0,T_{\max})$
with $T_{\max}\in(0,\infty)$. 
 If
 \begin{align}\label{lem2:con1}
\mathop{\sup}\limits_{t\in (0,T_{\max})} \norm[L^\frac{N}{2}(\Om)] {u(\cdot,t)}<\infty, \\ \mbox{ and }
\{u^\frac{N}{2}(\cdot,t)\}_{t\in(0,T_{\max})}
 \mbox{ is equi-integrable}.
 \end{align}
 Then there is $p\in(\frac N2,N)$ such that
  \begin{align}\label{lem2:lp}
  \mathop{\sup}\limits_{t\in (0,T_{\max})} \norm[L^p(\Om)] {u(\cdot,t)}<\infty.
 \end{align}
\end{lem}

\begin{proof}
Let $p\in(\frac N2,N)$.
Let $\theta\in(1,\infty)$ 
satisfy $\frac{1}{\theta}=1+\frac 2N-\frac 2{p} \in(0,1)$, and $\theta'$ 
be such that $\frac{1}{\theta}+\frac{1}{\theta'}=1$.
We test the first equation in \eqref{eq:system} with $pu^{p-1}$ to obtain that
\begin{align}\nn
  \frac{d}{dt}\io u^p+p(p-1)\io u^{p-2}|\nabla u|^2
  &=p(p-1)\io u^{p-1}\nabla u\cdot\nabla v\\\nn
  &\le \frac{p(p-1)}{4} \io u^{p-2}|\nabla u|^2+p(p-1)\io u^p|\nabla v|^2
\end{align}
for all $t\in(0,T_{\max})$.
Applying H\"older's inequality, we get
\begin{align}\label{eq:lem2:1}
\frac{d}{dt} \io u^p+ \frac{3(p-1)}{p}\io |\nabla u^\frac{p}{2}|^2\le p(p-1)\io  u^p|\nabla v|^2
\le p(p-1)
\left(\io u^{p\theta}\right)^\frac{1}{\theta}
\left(\io |\nabla v|^{2\theta'}\right)^\frac{1}{\theta'}
\end{align}
for all $t\in(0,T_{\max})$.
Let $a:=\frac{p-\frac{N}{2\theta}}{1-\frac N2+p}\in(0,1)$, 
and abbreviate $\frac{1}{1-a}=:\lambda>1$. The Gagliardo-Nierenberg inequality implies the existence of $c_1>0$ such that
\begin{align}\nn
 p(p-1)\left(\io u^{p\theta}\right)^\frac {1}{\theta}
 &=(p-1)\norm[L^{2\theta}(\Om)]{u^\frac{p}{2}}^2
 \le c_1\norm[L^2(\Om)]{\nabla u^{\frac p2}}^{2a}
 \norm[L^{\frac Np } (\Om)]{u^\frac p2}^{2(1-a)}
 +c_1\norm[L^{\frac{N}{p}}(\Om)]{u^\frac{p}{2}}^2.
\end{align}
Using Young's inequality and the assumption \eqref{lem2:con1}, we find some
constant $c_2>0$ such that
 the right-hand side of \eqref{eq:lem2:1} is estimated as
\begin{align}\nn
p(p-1)\left(\io u^{p\theta}\right)^\frac{1}{\theta}
\left(\io |\nabla v|^{2\theta'}\right)^\frac{1}{\theta'}
 &\le 
 ( c_1\norm[L^2(\Om)]{\nabla u^{\frac p2}}^{2a}+c_1\norm[L^{\frac{N}{p}}(\Om)]{u^\frac{p}{2}}^2)
 \norm[L^{2\theta'} (\Om)]{\nabla v}^2\\\label{eq:lem2:2}
 &\le \frac{p-1}{p}  \norm[L^2(\Om)]{\nabla u^{\frac p2}}^2
 +c_2\norm[L^{2\theta'}(\Om)]{\nabla v}^{2\lambda}+c_2.
\end{align}
Due to the choices of $\theta$ and $\theta'$, we know that $p\in(1,N)$
and $2\theta'=\frac{Np}{N-p}$, hence an application of Lemma \ref{lem:embedding} yields $c_3>0$ such that
\begin{align}\label{eq:lem2:3}
\norm[L^{2\theta'}(\Om)]{\nabla v}^{2\lambda}
\le
c_3\norm[L^p(\Om)]{\Delta v}^{2\lambda}+c_3\norm[L^1(\Om)]{v}^{2\lambda}.
\end{align}

We also recall from the Gagliardo-Nirenberg inequality that there is $c_4>0$ fulfilling
\begin{align}\label{eq:lem2:7}
\frac{p-1}{p}\io|\nabla u^\frac{p}{2}|^2\ge \lambda\io u^p-c_4.
\end{align}
Thus we conclude from the previous estimates (\ref{eq:lem2:1}-\ref{eq:lem2:7}) 
and Lemma \ref{lem:l1} that
\begin{align}
 \frac{d}{dt}\io u^p+{\lambda}\io u^p+ \frac{(p-1)}{p}\io |\nabla u^\frac p2|^2
 \le c_3\norm[L^p(\Om)]{\Delta v}^{2\lambda}+c_2+c_4 +c_3\norm[L^1(\Om)]{v}^{2\lambda} \fat.
 \end{align}

Let $t_0\in(0,T_{\max})$. Applying the variation-of-constants formula to the above inequality, 
we find a constant $c_5>0$ such that
\begin{align}\nn
 \io u^p(\cdot,t)&\le e^{-{\lambda}(t-t_0)}\io u^p(\cdot,t_0)
 -\frac{(p-1)}{p}
 \int_{t_0}^t e^{-\lambda(t-s)}
 \io |\nabla u^\frac p2(\cdot,s)|^2 ds\\
 \label{eq:lem2:5}
 &\quad 
 +c_3\int_{t_0}^t e^{-\lambda(t-s)}\norm[L^p(\Om)]{\Delta v(\cdot,s)}^{2\lambda} ds
 +c_5
\end{align}
for all $t\in(t_0,T_{\max})$.
The maximal regularity from Lemma \ref{chap1:lem:maximal} provides a constant 
$c_6>0$ satisfying
\begin{align}
 c_3\int_{t_0}^t e^{-\lambda(t-s)}\norm[L^p(\Om)]{\Delta v}^{2\lambda} ds\le 
 c_6\int_{t_0}^t  e^{-\lambda(t-s)}\norm[L^p(\Om)]{u}^{2\lambda} ds+c_6.
\end{align}




Let $d=\frac{p-\frac N2}{1-\frac N2+p}$ and 
$b=\frac{\frac pN-\frac12}{\frac pN-\frac{p}{2p+2}}$. We can easily check that
$\frac{4\lambda}{p} d=2$. 
Since $\{u^\frac{N}{2}(\cdot,t)\}_{t\in(0,T_{\max})}$ is uniformly integrable, 
and therefore belongs to the set $\mathcal{F}_\delta$ defined in 
\eqref{chap1:f_delta} with some nondecreasing $\delta: (0,1)\to(0,\infty)$. 
Since (\ref{lem2:con1}), with  
\[\ep:=\frac{\frac{p-1}{p}}{\mathop{\sup}\limits_{t\in(0,T_{\max})}
\norm[L^\frac{N}{p}(\Om)]{u^\frac{p}{2}}^{\frac{4\lambda}{p}(1-b)}}>0,
\]
we can apply 
Lemma \ref{lem:interpolation1} (in the case $q=r=2$, and with $\theta=\frac{N}{p}<q$ 
by virtue of 
$p>\frac N2$)
to find $c_\ep>0$ such that
\begin{align}\nn
 c_6\norm[L^p(\Om)]{u}^{2\lambda}
 &=c_6\norm[L^2(\Om)]{u^\frac{p}{2}}^{\frac{4\lambda}{p}}\\
 &\le \ep c_6\norm[L^2(\Om)]{\nabla u^\frac{p}{2}}^{\frac{4\lambda}{p} d}
 \norm[L^\frac{N}{p}(\Om)]{u^\frac{p}{2}}^{\frac{4\lambda}{p}(1-b)}
+c_\ep
 \le \frac {(p-1)}{p} \norm[L^2(\Om)]{\nabla u^\frac{p}{2}}^2+c_\ep
\end{align}
for all $t\in(0,T_{\max})$,
which leads to
\begin{align}\label{eq:lem2:6}
c_3\int_{t_0}^t e^{-\lambda(t-s)}\norm[L^p(\Om)]{\Delta v(\cdot,s)}^{2\lambda} ds
\le 
\frac {(p-1)}{p}\int_{t_0}^t e^{-\lambda(t-s)} 
\io|\nabla u^\frac p2(\cdot,s)|^2 ds+c_\ep+c_6
\end{align} 
for all $t\in(t_0,T_{\max})$.
Adding this to \eqref{eq:lem2:5} 
shows that
\[
 \io u^p(\cdot,t)
 \le e^{-\lambda(t-t_0)}\io u^p(\cdot,t_0)
 +c_5+c_6+c_\ep\le \io u^p(\cdot,t_0)
 +c_5+c_6+c_\ep
\]
for all $t\in(t_0,T_{\max})$. 
Since $\mathop{\sup}\limits_{t\in(0,t_0]} \norm[\Lp]{u(\cdot,t)}<\infty$
 due to the local existence theory, this shows
 \eqref{lem2:lp}.
 \end{proof}

\begin{proof}[Proof of Theorem \ref{thm:bdd}]
Employing Lemma \ref{lem:n} and Proposition \ref{lem:lp} proves 
$\mathop{\sup}\limits_{t\in(0,T_{\max})} \norm[\Lin]{u(\cdot,t)}<\infty$, which 
combined with Lemma \ref{lem:locexist} implies that $T_{\max}=\infty$. Thus the solution is global
and bounded.
\end{proof}

\section{Blow up behavior}\label{sec:blowup}

From another aspect, the extension criterion in Theorem 
\ref{thm:bdd} also gives the
characterization of blow up solutions.

\begin{proof}[Proof of Theorem \ref{cor:blowup}]
Suppose on contrary that 
$\{u^\frac{N}{2}(\cdot,t)\}_{t\in(0,T_{\max})}$ 
is equi-integrable  with $T_{\max}\in(0,\infty]$. 
We can apply Theorem \ref{thm:bdd} to show that there is a constant $C>0$ such that
\[\norm[\Lin]{u(\cdot,t)}\le C,\] 
 for all $t\in(0,T_{\max})$, which is a contradiction.
\end{proof}

\renewcommand\thesection{\Alph{section}}
\setcounter{section}{0}
\section{Appendix}
We claim a basic property of extension functions which we have used in the proof of 
Lemma \ref{lem:interpolation1}. 
Namely, the extension function $\varphit\in W^{1,r}(\Om')$ 
is equi-integrable with respect to some power in $\Om'$ provided $\varphi$ has the same property 
in $\Om$. Since we can not find this precise result
in any reference,
we also give a brief proof here.
\begin{thm}\label{thm:extension}
Assume that $\Om\subset \R^N$ is a bounded domain with smooth boundary 
and that $r>1$, $1\le q<\frac{Nr}{(N-r)_+}$. Let $\Om'$ be a bounded smooth domain
with $\Om\subset\Om'$. Then there is $C>0$ and for any nondecreasing function 
$\delta:(0,1)\to
(0,\infty)$, we can find $\tilde\delta:(0,1)\to
(0,\infty)$ nondecreasing such that we can extend any function $\varphi\in W^{1,r}(\Om)$ to a
function $\varphit\in W^{1,r}_0(\R^N)$
in such a way that 
 \begin{align}\label{chap1:app1}
 &{\varphit}=\varphi \mbox{ a.e.  in }\Om,\quad \supp {\varphit} \subset \Om',\\
 \label{chap1:app2}
 & \norm[W^{1,r}(\Om')]{\nabla{\varphit}}^r
 \le C\norm [W^{1,r}(\Om)]{{\nabla\varphi}}^r,\\\label{chap1:app3}
 &\norm[L^q(\Om')]{{\varphit}}\le C\norm [L^q(\Om)]{{\varphi}}.\end{align}
Moreover, if  $\varphi\in\mathcal{F}_\delta$ with 
\begin{align}\nn
\mathcal F_{\delta}:=\bigg\{ \psi\in W^{1,r}(\Om)\ \bigg| \ 
\mbox{ For all } \ep'\in(0,1), &\mbox{ we have } \int_{E} \psi^p<\ep'  \mbox{ for all 
measurable sets }\\\label{phi_delta}&
E \subset \Om \mbox{ with } |E|<\delta(\ep')
 \bigg\},
\end{align}
then 
$\varphit\in \mathcal{F_{\widetilde\delta}}$ with
\begin{align}\nn
\mathcal F_{\widetilde\delta}:=\bigg\{ \psi\in W^{1,r}(\Om')\ \bigg| \ 
\mbox{ For all } \ep'\in(0,1), &\mbox{ we have } \int_{E} \psi^p<\ep'  
\mbox{ for all measurable sets }\\
\label{phi'_delta'}&
E \subset \Om' \mbox{ with } |E|<\widetilde\delta(\ep')
 \bigg\}.
\end{align}
\end{thm}
\begin{proof}

First, \eqref{chap1:app1} and \eqref{chap1:app2} are precisely proven in
\cite[Theorem 5.4.1]{evans}.
Now we recall the construction of the extension function in the proof to show 
the remaining properties.
Since $\pa\Om$ is compact, we can find finitely many points 
$\{x_i\}_{1\le i\le K}\subset \pa\Om$ and open sets $\{W_i\}_{1\le i\le K}\subset \Om'$ 
with $x_i\in W_i$ and $W_0\subset \Om$ such that $\pa\Om\subset\mathop{\cup}\limits_{1\le i\le K} W_i$ and 
$\Om\subset W_0\cup(\mathop{\cup}\limits_{1\le i\le K} W_i)\subset \Om'$. There exist $C^1$ 
diffeomorphisms $\Phi_i: W_i\to \R^N$ $(1\le i\le K)$ which flatten out $\pa\Om$ near $x_i$; 
namely, if we let 
$B_i:=\Phi_i(W_i)$ be a ball, it satisfies 
$B_i^-=\Phi_i(W_i\cap\Om^c)=\{y=(y_1,...,y_N)|~y_N<0\}$, 
$B_i^+=\Phi_i(W_i\cap\Om)=\{y=(y_1,...,y_N)|~y_N>0\}$.
Now we define 
linear transformations 
\begin{align*}
&Y_1:(y_1,...,y_N)\in B_i^-\to (y_1,...,y_{N-1},-y_N)\in B_i^+,
\\
&Y_2:(y_1,...,y_N)\in B_i^-
\to (y_1,...,y_{N-1},-\frac12 y_N) \in B_i^+.
\end{align*} 
Let
$\varphi_i'(y)=\varphi(\Phi_i^{-1}(y))$ ($y\in B_i+$, $x=\Phi_i^{-1}(y)\in W_i\cap\Om$).
A first order reflection of $\varphi_i'(y)$ is given by 
\begin{equation}
\varphit_i'(y):=
\left\{
\begin{array}{lll}
&-3\varphi_i'(Y_1(y))+4\varphi_i'(Y_2(y)), & y\in B_i^-,
\\
& \varphi_i'(y), & y\in B_i^+.
\end{array}
\right.
\end{equation}
If we let $\{\zeta_i\}_{0\le i\le K}$ be a partition of unity subordinate to $\{W_i\}_{0\le i\le K}$,
the associated extension $\varphit: \Om'\to \R^N$ of $\varphi$ is defined by converting $\varphit_i'$ 
back to $W_i$
\begin{equation}
\varphit(x):=
\left\{
\begin{array}{lll}
& \varphi(x), & x\in \Om=\mathop{\cup}\limits_{0\le i\le K} W_i+,\\[6pt]
& \mathop{\sum}\limits_{i=0}^{i=K} \zeta_i(x) 
\left\{-3\varphi(\Phi_i^{-1}(Y_1(\Phi_i(x))))+4\varphi(\Phi_i^{-1}(Y_2(\Phi_i(x))))\right\},
&x\in \mathop{\cup}\limits_{1\le i\le K} W_i-,\\[6pt]
& 0, & x\in \Om' \backslash\mathop{\cup}\limits_{0\le i\le K} W_i,
\end{array}
\right.
\end{equation}
where $W_i^+:=\Phi_i^{-1}(B_i^+)$, $W_i^-:=\Phi_i^{-1}(B_i^-)$.
Since the mappings $\Phi_i$, $\Phi_i^{-1}$ $(1\le i\le K)$, $Y_j$ $(j=\{1,2\})$ are $C^1$, 
we can find a constant $c_1>0$
such that $|\Phi_i^{-1}(Y_i(\Phi_i(U)))|\le c_1 |U|$ for all $U\subset W_i^-$ $(1\le i\le K)$.
For any measurable subset $E'\subset \Om'$, let
$E_i:=E'\cap W_i^-$.
We note that $\Phi_i^{-1}(Y_2(\Phi_i(E_i)))\subset\Phi_i^{-1}(Y_1(\Phi_i(E_i)))\subset
\Phi_i^{-1}(B_i^+)\subset \Om$.
By changing variables, for each $1\le i\le K$, we have
\begin{align}\nn 
\int_{E_i} |\varphit(x)|^p dx
&=\int_{E_i}  |-3\varphi(\Phi_i^{-1}(Y_1(\Phi_i(x))))
+4\varphi(\Phi_i^{-1}(Y_2(\Phi_i(x))))|^p dx\\\nn
&=\int_{\Phi_i(E_i)} 
|-3\varphi(\Phi_i^{-1}(Y_1(y)))
+4\varphi(\Phi_i^{-1}(Y_2(y)))|^p |\det (D\Phi_i^{-1}(y))| dy\\\nn
&=\int_{\Phi_i(E_i)} |-3\varphi_i'(y_1,...,y_{N-1},-y_n)+4
\varphi_i'(y_1,...y_{n-1},-\frac 12 y_n)|^p |\det (D\Phi_i^{-1}(y))| dy\\\nn
&\le 2^{p-1}\int_{Y_1(\Phi_i(E_i))} 3^p |\varphi'_i(y)|^p |\det (D\Phi_i^{-1}(y))| dy
+2^{p-1}\int_{Y_2(\Phi_i(E_i))} 4^{p-1}\frac 12 |\varphi'_i(y)|^p |\det (D\Phi_i^{-1}(y))|dy\\\nn
&\le 6^{p}\int_{\Phi_i^{-1}(Y_1(\Phi_i(E_i)))}|\varphi(x)|^p dx
+8^{p}\int_{\Phi_i^{-1}(Y_2(\Phi_i(E_i)))} |\varphi(x)|^p  dx
1\end{align}
According to \eqref{phi_delta},
given $\ep'>0$, we have that $\delta(\ep')>0$ such that 
$\int_E \varphi^p<\frac{\ep'}{8^p(3K)}$ 
for all $E\subset \Om$ with $|E|\le \delta(\ep')$. 
We let $\widetilde\delta:=\frac{1}{c_1}\delta$ such that if $|E'|<\min\{\widetilde\delta,\delta\}$, 
then $|\Phi_i^{-1}(Y_1(\Phi_i(E_i)))|, |\Phi_i^{-1}(Y_2(\Phi_i(E_i)))|<\delta$ for all $1\le i\le K$, hence
\begin{align}\nn
\int_{E'} |\varphit(x)|^p dx&= \int_{E'\cap\Om} |\varphi(x)|^pdx+\int_{E'\cap\Om^c} |\varphit(x)|^p dx\\\nn
&\le \int_{E'\cap\Om} |\varphi(x)|^p dx+\mathop{\sum}\limits_{i=1}^{i=K}
\int_{E_i} |\varphi(x)|^p dx\\\nn
&\le \int_{E'\cap\Om} |\varphi(x)|^p dx+\mathop{\sum}\limits_{i=1}^{i=K}
\left(6^{p}\int_{\Phi_i^{-1}(Y_1(\Phi_i(E_i)))} |\varphi(x)|^p dx
+8^{p}\int_{\Phi_i^{-1}(Y_2(\Phi_i(E_i)))} |\varphi(x)|^pdx
\right)\\\nn
&\le \frac{\ep'}{8^p 3K}+K(\frac{6^p\ep'}{8^p 3K}
+\frac{8^{p}\ep'}{8^p 3K})< \ep'.
\end{align}
Therefore, $\varphit\in \mathcal F_{\widetilde\delta}$ is shown.
Using $\int_{\Om'} |\varphit|^q=\mathop{\sum}\limits_{0\le i\le K} \int_{E_i} |\varphit|^q$, \eqref{chap1:app3} can be proven in a similar way. 
\end{proof}

\section*{Acknowledgement}
The author would like to thank Johannes Lankeit, Christian Stinner and Michael Winkler for carefully reading and for 
many useful comments, which significantly improve the paper.

\end{document}